\documentclass[a4paper,12pt]{amsart}
\usepackage{amsthm,amsmath, amssymb, amscd,mathtools, braket}
\usepackage[top=30truemm, bottom=30truemm, left=25truemm, right=25truemm]{geometry}
\usepackage{color}
\usepackage{stmaryrd}
\usepackage{bbm}
\usepackage[dvipdfmx,setpagesize=false]{hyperref}
\usepackage{mathrsfs}

\numberwithin{equation}{section}
\newcounter{count}

\newcommand{\num}{\stepcounter{count}\the\value{count}}

\renewcommand{\limsup}{\varlimsup}

\newtheorem{theorem}{Theorem}[section]
\newtheorem{lemma}[theorem]{Lemma}
\newtheorem{corollary}[theorem]{Corollary}

\newtheorem{proposition}[theorem]{Proposition}

\newtheorem{question}[theorem]{Question}

\theoremstyle{remark}
\newtheorem{remark}[theorem]{Remark}

\begin{document}

\title[Mills' constant is irrational]{Mills' constant is irrational}

\author[K. Saito]{Kota Saito}
\address{Kota Saito\\Department of Mathematics\\ College of Science $\&$ Technology \\ Nihon University\\Kanda\\ Chiyoda-ku\\ Tokyo\\
101-8308\\ Japan}
\email{saito.kota@nihon-u.ac.jp}

\thanks{It will appear in \textit{Mathematika}.}

\subjclass[2020]{11J72, 11J81}
\keywords{prime-representing constants, Mills' constant, transcendental number theory}

\begin{abstract}
Let $ \lfloor x \rfloor $ denote the integer part of $ x $. In 1947, Mills constructed a real number $ \xi > 1 $ such that $\lfloor \xi^{3^k} \rfloor$ is always a prime number for every positive integer $k$. We define Mills' constant as the smallest real number $\xi$ satisfying this property. Determining whether this number is irrational has been a long-standing problem. In this paper, we show that Mills' constant is irrational. Furthermore, we obtain partial results on the transcendency of this number.
\end{abstract}

\maketitle

%\vspace*{6pt}\tableofcontents  % for this guide only.
% A table of contents should normally not be included

\section{Introduction}
Let $\mathbb{N}$ denote the set of all positive integers. Let $\lfloor x\rfloor$ denote the integer part of $x\in \mathbb{R}$. Mills \cite{Mills} constructed a real number $\xi>1$ such that 
\begin{equation} \label{condition-Mills}
\text{$\lfloor  \xi^{3^k}\rfloor$ is a prime number for every $k\in \mathbb{N}$}.
\end{equation}
Such $\xi$ is not uniquely determined. In this paper, we define \textit{Mills' constant} \cite[p.130]{Finch} as the smallest $\xi>1$ satisfying \eqref{condition-Mills}. In Lemma~\ref{Lemma:Existence}, we will show the existence of the smallest element of such $\xi$'s. Assuming the Riemann hypothesis, Caldwell and Cheng \cite{CaldwellCheng} determined $6850$ digits of the decimal expansion of Mills' constant which is approximately equal to $1.3063778838\cdots$. Proving the irrationality of a specific number satisfying \eqref{condition-Mills} has remained an unsolved problem for over 70 years since Mills' result. In this paper, we determine that Mills' constant is irrational.

Before stating the results, for every integer $c\geq 2$, we define $\xi_c$ as the smallest real number $\xi>1$  such that $\lfloor \xi^{c^k} \rfloor$ is a prime number for every $k\in \mathbb{N}$. The number $\xi_c$ exists for every integer $c\geq 2$ by Corollary~\ref{Corollary:Matomaki} (or \cite[Corollary 4]{Matomaki} and \cite[Remark~4.2]{SaitoTakeda}). A real algebraic integer greater than $1$ is a \textit{Pisot number} if all of its conjugates over $\mathbb{Q}$ except for itself lie in the open unit disc. We obtain the following results. 

\begin{theorem}\label{Theorem-MillsType}
For every integer $c\geq 4$, $\xi_c$ is transcendental.  
\end{theorem}
\begin{theorem}\label{Theorem-Mills}
Either $\xi_3$ is transcendental, or there exists a positive integer $m$ such that $\xi_3^{3^m}$ is a Pisot number of degree $3$. In particular, $\xi_3$ is irrational.  
\end{theorem}

For more generalization, let $(c_k)_{k=1}^\infty$ be a sequence of real numbers with $c_{k}\geq 1$ for all $k\in \mathbb{N}$, and we define 
\[
\mathcal{W}(c_k)=\mathcal{W}((c_k)_{k=1}^\infty)= \{\xi >1\colon \lfloor \xi^{C_k} \rfloor \text{ is a prime number for every $k\in \mathbb{N}$}\}, 
\]
where $C_k=c_1\cdots c_k$ for all $k\in \mathbb{N}$.   
\begin{remark}
We always use the notation $\mathcal{W}(c_k)$ as $\mathcal{W}((c_k)_{k=1}^\infty)$ for simplicity, and so we remark that $\mathcal{W}(c_k)$ is independent of $k$. For instance, if $(c_k)_{k=1}^\infty=(3,3,3,\ldots)$, then
\[
\mathcal{W}(3)= \{\xi >1\colon \lfloor \xi^{3^k} \rfloor \text{ is a prime number for every $k\in \mathbb{N}$}\}. 
\]
\end{remark}
 An element $\xi \in \mathcal{W}(c_k)$ is called a \textit{prime-representing constant} (for short \textit{PRC}). After Mills' result \cite{Mills}, there are several studies for extending this result. For example, in 1950, Ku{\i}pers \cite{Kuipers} constructed an element in $\mathcal{W}(c)$ for all integers $c\geq 3$. In 1951, Ansari \cite{Ansari} extended the range of $c$ to all real numbers $c>77/29=2.655\cdots$. In the same period, Niven \cite{Niven} independently gave a similar extension, but he extended $c$ to real numbers $c>8/3=2.666\cdots$. Furthermore, in 1954, Wright \cite{Wright54} first studied the set of PRCs\footnote{In honor of Wright, the author uses the symbol ``$\mathcal{W}$" as the set of prime-representing constants. Wright actually considered a wider class of representing functions.}. We refer the readers to Dudley's survey \cite{Dudley} for more details. It is a good survey of the early research of Mills-type constants.

% Wright \cite{Wright51} constructed a %real number $\mu>1$ such that all of
%\begin{equation}\label{eq:sequence}
%\lfloor 2^\mu\rfloor,\quad  \lfloor 2^{2^\mu}\rfloor,\quad  \lfloor 2^{2^{2^\mu}}\rfloor, \quad \ldots
%\end{equation}
%are prime numbers by applying Bertrand's postulate (or the Bertrand-Chebyshev theorem) which asserts that each interval $(n,2n]$ $(n=1,2,\ldots)$ contains a prime number. We note that by setting $\phi_1=x^3$, the sequence $\lfloor A^{3}\rfloor, \lfloor A^{3^2}\rfloor, \lfloor A^{3^3}\rfloor, \cdots$ can be represented by $\lfloor \phi_1(A)\rfloor, \lfloor \phi_1\circ\phi_1 (A)\rfloor, \lfloor \phi_1\circ \phi_1\circ\phi_1 (A)\rfloor,\cdots$, where $f\circ g(x)=f(g(x))$. Furthermore,  setting  $\phi_2=2^x$, the sequence \eqref{eq:sequence} also can be represented by the iterated composition of $\phi_2$. In 1952, Ore generalized such construction by using the iterated composition of functions, and 
%in 1954, %Wright \cite{Wright54} gave a wide class of representing functions. Remarkably, Wright did not only generalize Mills' result, but he first started to study sets of representing constants such as $\mathcal{W}(C_k)$. 

 Especially, by Wright's result \cite[Theorem~5]{Wright54}, $\mathcal{W}(c_k)$ is uncountable under suitable conditions on $(c_k)_{k=1}^\infty$. Therefore, there exist infinitely many transcendental numbers in $\mathcal{W}(c_k)$, but we do not get any arithmetic properties of the smallest element in $\mathcal{W}(c_k)$ from this result. Alkauskas and Dubickas \cite[Theorem~1]{AlkauskasDubickas} gave a result on the transcendency of specific PRCs. They succeeded in constructing a transcendental number in $\mathcal{W}(c_k)$ if the sequence $(c_k)_{k=1}^\infty$ satisfies
\begin{center}
(i) $c_1=1$; \quad (ii) $c_{k+1}>2.1053$ and $C_k\in \mathbb{N}$ for all $k\in \mathbb{N}$; \quad (iii) $\limsup_{k\rightarrow \infty} c_{k+1}=\infty$. 
\end{center}
They did not discuss whether the smallest element in $\mathcal{W}(c_k)$ is transcendental. The author and Takeda \cite[Theorem~1.3]{SaitoTakeda} showed the following result. 
\begin{theorem}\label{Theorem:STakeda}
Let $(c_k)_{k=1}^\infty$ be a sequence of positive integers satisfying 
\begin{enumerate}\renewcommand{\theenumi}{\roman{enumi}}
\renewcommand{\labelenumi}{(\theenumi)} 
\item $c_1\geq 1$\textup{;}  
\item $c_{k+1}\geq 2$ for all $k\in \mathbb{N}$\textup{;}
\item \label{Condition:STakeda}$\limsup_{k\rightarrow \infty} c_{k+1}=\infty$.
\end{enumerate}
Then, the smallest element in $\mathcal{W}(c_k)$ exists and it is transcendental.
\end{theorem}
The author and Takeda also obtained a result on the algebraic independency of PRCs in \cite{SaitoTakeda}. The condition \eqref{Condition:STakeda} is essentially needed in these works, but we accomplish replacing \eqref{Condition:STakeda} with the condition $c_{k+1}\geq 4$ for infinitely many positive integers $k$ as follows.

\begin{theorem}\label{Theorem-main1} Let $b$ be an integer greater than or equal to $3$. Let $(c_k)_{k\in\mathbb{N}}$ be a sequence of integers satisfying
\begin{enumerate}
\item \label{Condition-c1}$c_1\geq 1$\textup{;}
\item \label{Condition-c2}$c_{k+1}\geq 2$ for all $k\in \mathbb{N}$\textup{;}
\item \label{Condition-c3}$c_{k+1}=  b$ for infinitely many positive integers $k$.
\end{enumerate}
Then the smallest real number in $\mathcal{W}(c_k)$ exists, say $\xi$. For this $\xi$ the following is true:
\begin{itemize}
\item if $b\geq 4$, then $\xi$ is transcendental;
\item if $b=3$, then either $\xi$ is transcendental, or $\xi ^{C_m}$ is a Pisot number of degree $3$ for some $m\in \mathbb{N}$.
\end{itemize}
\end{theorem} 

Theorem~\ref{Theorem-main1} immediately implies Theorems~\ref{Theorem-MillsType} and \ref{Theorem-Mills}. We will prove Theorem~\ref{Theorem-main1} in Section~\ref{Section-Proof}. At the end of this section, let us ask several remaining questions.

\begin{question}\label{Question-1}
Is $\xi_2$ rational or irrational?
\end{question}
\begin{question}\label{Question-2}
Does there exist $m\in \mathbb{N}$ such that $\xi_3^{3^m}$ is a Pisot number of degree $3$?
\end{question}

If the answer to Question~\ref{Question-2} is ``no'', then $\xi_3$ is transcendental by Theorem~\ref{Theorem-Mills}. In Section~\ref{Section-Further}, we will give assumptions leading to the transcendency of $\xi_3$. We will explain the main difficulty of Question~\ref{Question-2} in Remark~\ref{Remark-b=3}.

\section{Strategy of our proof}\label{Section-Strategy}

In this section, we demonstrate a strategy of our proof. We first recall Mills' construction, that is, we now construct $\xi>1$ such that $\lfloor \xi^{3^k} \rfloor$ is a prime number for every $k\in \mathbb{N}$. 

Let $\mathcal{P}$ be the set of all prime numbers, and let $\theta=21/40$. We apply the following theorem given by  Baker, Harman, and Pintz \cite[p.562]{BakerHarmanPintz}\footnote{Mills applied the result given by Ingham \cite{Ingham} which states that $p_{n+1}-p_n < Kp_n^{5/8}$ for some positive constant $K$. Thus, the construction of this section is modernly modified. }.
\begin{theorem}\label{Theorem-BakerHarmanPintz} There exists a real number $d_0>0$ such that
\[
 \# ([x, x+x^{\theta}] \cap \mathcal{P})  \geq \frac{d_0 x^\theta}{\log x}
 \]
for sufficiently large $x>0$.
\end{theorem} 
Let $p_1$ be a sufficiently large prime number. By Theorem~\ref{Theorem-BakerHarmanPintz} with $x=p_1^3$, we find $p_2\in \mathcal{P}$ such that 
\[
p_1^3 \leq p_2 \leq p_1^3 +p_1^{3\theta}. 
\] 
By $3\theta= 63/40<2$, we see that 
\[
p_1^3 +p_1^{3\theta}< p_1^3+3p_1^2 +3p_1 = (p_1+1)^3-1.
\]
Similarly,  by Theorem~\ref{Theorem-BakerHarmanPintz} with $x=p_2^3$,  we find $p_3\in \mathcal{P}$ such that 
\[
p_2^3 \leq p_3 \leq p_2^3 +p_2^{3\theta}< (p_2+1)^3-1 . 
\] 
By iterating this argument, we find a sequence $(p_k)_{k=1}^\infty$ of prime numbers such that 
\begin{equation}\label{eq:cubic}
p_k^3 \leq p_{k+1} \leq  p_k^3 +p_k^{3\theta}< (p_k+1)^3-1 
\end{equation}
for every $k\in \mathbb{N}$. This leads to 
\[
p_1^{1/3^1} \leq p_2^{1/3^2} \leq p_3^{1/3^3} \leq \cdots < (p_3+1)^{1/3^3} < (p_2+1)^{1/3^2}<(p_1+1)^{1/3^1},
\]
and hence $\lim_{k\to\infty} p_k^{1/3^k}=\lim_{k\to\infty} (p_k+1)^{1/3^k} \eqqcolon w$ exists. Therefore, for every $k\in \mathbb{N}$, we have $p_k \leq w^{3^k} <p_k+1$, which implies that $p_k = \lfloor w^{3^k} \rfloor$ for every $k\in \mathbb{N}$. \\

We will further show the following result for giving a sketch of the proof.  
\begin{theorem}\label{Theorem-simple}
The number $w$ is irrational. 
\end{theorem}

\begin{remark}
We actually want to discuss the arithmetic properties of $\xi_3(=\min \mathcal{W}(3))$. There is a gap between $w$ and $\xi_3$, but we can apply a similar way of $w$ to $\xi_3$ (see Lemmas~\ref{Lemma-key1} and \ref{Lemma-key2}).  
\end{remark}

To prove Theorem~\ref{Theorem-simple}, we estimate the distance between $w^{3^k}$ and the nearest integer. 

\begin{lemma}\label{Lemma-simpleNear}
There exists $\gamma>0$ such that for every sufficiently large $k$ 
\[
|w^{3^{k}} -p_k|  \leq e^{-3^{k} \gamma}.
\]
\end{lemma}

\begin{proof}
Since $p_{k+1}=\lfloor w^{3^{k+1}}\rfloor$ for every $k\in \mathbb{N}$,  \eqref{eq:cubic} leads to 
\[
p_k^3 \leq p_{k+1} \leq w^{3^{k+1}}  \leq  p_k^3 +p_k^{3\theta}+1\leq p_k^3 (1 +2 p_k^{3(\theta-1)} )
\]
for every $k\in \mathbb{N}$. Therefore,  
\[
p_k \leq w^{3^{k}}  \leq p_k  (1+2p_k^{3(\theta-1)})^{1/3}< p_k +2p_k^{3\theta-2}, 
\]
where the last inequality follows from $(1+x)^{1/3}<1+x$ for all $x>0$. 
%By applying the Maclaurin expansion, $(1+x)^{1/3}=1+O(x)$ holds if $|x|\leq 1$. Thus, since $\theta-1=-19/40<0$, for every sufficiently large $k$, we have 
%\[
%p_k \leq w^{3^{k}}  \leq p_k \left(1+O(p_k^{3(\theta-1)})\right) = p_k + O(p_k^{3(\theta-1)+1} ),
%\]
Thus, for every $k\in \mathbb{N}$, we have 
\[
 |w^{3^{k}} -p_k| \leq 2 p_k^{3\theta-2}.
\]
Combining $3\theta-2=63/40-2<0$ and the first inequality of \eqref{eq:cubic}, we obtain 
\[
|w^{3^{k}} -p_k| \leq 2 p_k^{3\theta-2}\leq 2 p_{k-1}^{3(3\theta-2)} \leq \cdots \leq 2 p_1^{3^{k-1} (3\theta -2) },
\]
which completes the proof of Lemma~\ref{Lemma-simpleNear} by selecting $\gamma=\frac{2-3\theta}{6}\log p_1>0$.
\end{proof}

Lemma~\ref{Lemma-simpleNear} means that $w^{3^{k}}$ is very close to some integer (especially, prime number). On the other hand, the powers of a rational number are at a distance from integers. Indeed, in 1957, Mahler \cite{Mahler} proved the following result. Here, for every $x\in \mathbb{R}$, we define 
\[
\|x\|= \min \{|x-m| \colon m\in \mathbb{Z} \}.
\]
\begin{theorem}\label{Theorem-Mahler} Let $\alpha$ be a rational number strictly greater than $1$ and which is not an integer. Let $\epsilon$ be a positive real number. Then, there exists an integer $n_0=n_0(\alpha,\epsilon)$ such that
\[
\|\alpha^n \| > e^{-\epsilon n}
\]
for every integer $n\ge n_0$.
\end{theorem}

\begin{proof}[Proof of Theorem~\ref{Theorem-simple}] It is clear that $w\notin \mathbb{Z}$  since if not, then we have $p_2 = \lfloor {w}^9\rfloor= (w^3)^3 =\lfloor w^3 \rfloor^3=p_1^3$, a contradiction. Therefore, $w$ is irrational by combining Lemma~\ref{Lemma-simpleNear} and Theorem~\ref{Theorem-Mahler} with $n\coloneqq 3^k$ and $\epsilon\coloneqq \gamma$.
\end{proof}

Through the discussion above, we can give a sketch of the proof of Theorem~\ref{Theorem-main1}. We will compare the following two phenomena:
\begin{itemize}
\item  powers of the smallest PRC are very close to some integer;
\item  powers of almost all real algebraic numbers are at a distance from integers. 
\end{itemize}
For observing the second phenomenon, we need a generalization of Theorem~\ref{Theorem-Mahler} to real algebraic numbers. Corvaja and Zannier \cite{CorvajaZannier} gave this generalization by applying the $p$-adic Schmidt subspace theorem. Dubickas used their result and prepared the following theorem to investigate a sequence $(x_n)_{n=0}^\infty$ defined in \eqref{eq:Recursive} (see \cite[Lemma~6]{Dubickas2022}). 
\begin{theorem}\label{Theorem-Dubickas1}Let $\alpha >1$ be an algebraic number, and $q$ be a positive integer. Suppose that $0<s_1<s_2<\cdots$ is a sequence of positive integers. Then, either for some $m\geq 0$ the number $\alpha^{s_m}$ is a Pisot number or for each $\epsilon>0$  there exists a positive integer $k_0 = k_0(\epsilon)$ such that 
\[
\| q\alpha^{s_k}  \|> e^{-\epsilon s_k }
\]
for every $k\geq k_0$.
\end{theorem}

Let $P(x)$ be a polynomial of degree $d\geq 2 $ with rational coefficients such that its leading coefficient is positive. Wagner and Ziegler \cite{WagneZiegler} studied a sequence $(x_n)_{n=0}^\infty$ satisfying 
\begin{equation}\label{eq:Recursive}
x_{n+1} = P(x_n)\quad \text{for $n=0,1,2\ldots$, \quad and $x_n\to \infty$ as $n\to \infty$.} 
\end{equation}
They showed that $\lim_{n\to \infty} x_n^{1/d^{n}}=\alpha>1$ exists and it is either irrational or an integer. Further, Dubickas \cite{Dubickas2022} improved their result to the transcendency of $\alpha$ by applying Theorem~\ref{Theorem-Dubickas1}. Their methods are instrumental not only for the recursive equations  \eqref{eq:Recursive} but also for recursive inequalities like \eqref{eq:cubic}. In addition, Dubickas presented the following lemma \cite[Lemma~8]{Dubickas2022}.
\begin{lemma}\label{lemma-Dubickas2}
Let $\beta$ be a Pisot number of degree $\ell \geq 2$ with conjugates $\beta_1=\beta, \beta_2,\ldots ,\beta_\ell$ labelled so that $\beta_1>1>|\beta_2|\geq  \cdots \geq |\beta_\ell|$. Then, there is a real number $\lambda>0$ depending only on $\beta$ such that
\[
|\beta_{2}^n +\cdots +\beta_{\ell}^n|\geq |\beta_2|^n n^{-\lambda} 
\]
for each sufficiently large integer $n$.
\end{lemma}
We will apply Lemma~\ref{lemma-Dubickas2} to the cases in which a Pisot number appears.  

\section{Lemmas and auxiliary results}\label{Section-Distance}

\begin{lemma}\label{Lemma:Existence}
Let $(c_k)_{k=1}^\infty$ be a sequence of positive real numbers with $c_k\geq 1$ for all $k\in \mathbb{N}$. If $\mathcal{W}(c_k)$ is non-empty, then the smallest element in $\mathcal{W}(c_k)$ exists. 
\end{lemma}
\begin{proof} Since $\mathcal{W}(c_k)$ is lower bounded and non-empty, $\inf \mathcal{W}(c_k)$ exists. Let $\xi=\inf \mathcal{W}(c_k)$.  By the definition of the infimum, there is a sequence $A_1\geq A_2 \geq \cdots $ such that $A_j\in \mathcal{W}(c_k)$ for every $j\in \mathbb{N}$ and $\lim_{j\to \infty} A_j=\xi$. By the right-side continuity of the floor function, for all fixed $k\in \mathbb{N}$, we have 
\[
\lfloor \xi^{C_k} \rfloor= \lfloor  \lim_{j\to \infty} A_j^{C_k} \rfloor  =  \lim_{j\to \infty} \lfloor   A_j^{C_k} \rfloor.
\] 
Thus, we have $\lfloor \xi^{C_k} \rfloor\in \mathcal{P}$ since $\lfloor   A_j^{C_k} \rfloor\in \mathcal{P}$ for all $j\in \mathbb{N}$ and $\lim_{j\to \infty} \lfloor A_j^{C_k}\rfloor$ converges in $\mathbb{Z}$. Therefore, $\inf \mathcal{W}(c_k) =\xi \in \mathcal{W}(c_k)$, which means that $\xi$ is the smallest element of $\mathcal{W}(c_k)$.  
\end{proof}

\begin{theorem}\label{Theorem:Matomaki}
Let $(c_k)_{k=1}^\infty$ be a sequence of real numbers satisfying $c_1>0$ and  $c_{k+1}\geq 2$ for all $k\in \mathbb{N}$. Then, $\mathcal{W}(c_k)$ is non-empty.
\end{theorem}
\begin{proof}
See \cite[Theorem~3]{Matomaki} or \cite[Theorem~3.3]{SaitoTakeda}.
\end{proof}

\begin{remark}
In \cite{Matomaki}, Matom\"{a}ki actually showed that $\mathcal{W}(c_k)$ is uncountable, nowhere dense, and has Lebesgue measure $0$ if $c_k\geq 2$ for all $k\in \mathbb{N}$. Thus, there is a small gap between the initial conditions $c_1>0$ in Theorem~\ref{Theorem:Matomaki} and $c_1\geq 2$ in Matom\"{a}ki's result, but it is not essential by considering the locally bi-Lipschitz continuous map 
\[
f_\alpha \colon \mathcal{W}(c_1,c_2,c_3,\ldots) \ni A \longmapsto A^{1/\alpha} \in \mathcal{W}(\alpha c_1,c_2,c_3,\ldots )
\]
for some fixed real number $\alpha>0$. We refer the readers to \cite[Remark~3.4]{SaitoTakeda} for more details.
\end{remark}

\begin{corollary}\label{Corollary:Matomaki}
Let $(c_k)_{k=1}^\infty$ be a sequence of real numbers satisfying $c_1>0$ and  $c_{k+1}\geq 2$ for all $k\in \mathbb{N}$. Then, the smallest element in $\mathcal{W}(c_k)$ exists.
\end{corollary}
\begin{proof}
It is clear from Lemma~\ref{Lemma:Existence} and Theorem~\ref{Theorem:Matomaki}. 
\end{proof}

Let $b$ be an integer greater than or equal to $3$. Let $(c_k)_{k=1}^\infty$ be a sequence of integers satisfying \eqref{Condition-c1}, \eqref{Condition-c2}, and \eqref{Condition-c3} in Theorem~\ref{Theorem-main1}. By Corollary~\ref{Corollary:Matomaki}, $\xi\coloneqq \min \mathcal{W}(c_k)$ exists. 

We may assume that $(c_k)_{k=1}^\infty$ is bounded, if not, $\xi$ is transcendental by Theorem~\ref{Theorem:STakeda}. Thus, we set 
\begin{equation}\label{equation:B}
B= \sup_{k\in \mathbb{N}} c_k <\infty,
\end{equation}
Furthermore, we let 
\begin{equation}\label{equation:Ib}
\mathcal{I}_b=\{k\in \mathbb{N}\colon c_{k+1} =b\}
\end{equation}
which is an infinite set from \eqref{Condition-c3}. Let $p_k=\lfloor \xi^{C_k} \rfloor$ for all $k\in \mathbb{N}$. Let $\theta=21/40$. The following Lemmas~\ref{Lemma-key1} and \ref{Lemma-key2} can be seen in \cite[Lemma~4.5 and Lemma~5.1]{SaitoTakeda}, but we give proofs for keeping the readability. 
\begin{lemma}\label{Lemma-key1}
For all $k\in \mathbb{N}$, we have 
\begin{equation}\label{eq:key1}
p_{k}^{c_{k+1}} \leq p_{k+1}< (p_{k}+1)^{c_{k+1}}-1.
\end{equation}
\end{lemma}

\begin{proof}Take any $k\in \mathbb{N}$. By the definition of $\xi$, we see that $p_k \leq \xi^{C_k} < p_{k}+1$, which implies that
$p_k^{c_{k+1}} \leq \xi^{C_{k+1}}<  (p_{k}+1)^{c_{k+1}}$. We recall that $c_{k+1}$ is a positive integer, and hence $p_k^{c_{k+1}} \leq \lfloor \xi^{C_{k+1}}\rfloor <  (p_{k}+1)^{c_{k+1}}$. It is not true that $p_{k+1}= (p_{k}+1)^{c_{k+1}}-1$ since $p_{k+1}$ is a prime number and $c_{k+1}$ is an integer greater than or equal to $2$. Therefore, we obtain \eqref{eq:key1}. 
\end{proof}

\begin{lemma}\label{Lemma-key2}
Let $\mathcal{I}_b$ be as in \eqref{equation:Ib}. There exists $k_0=k_0(\xi)>0$ such that for all $k\in \mathcal{I}_b\cap [k_0,\infty)$, we have 
\begin{equation}\label{eq:key2}
p_{k}^{c_{k+1}} \leq p_{k+1} \leq p_{k}^{c_{k+1}} + p_{k}^{\theta c_{k+1}}.
\end{equation}
\end{lemma}

Before proving Lemma~\ref{Lemma-key2}, we exhibit the following result given by Matom\"{a}ki  (see \cite[Lemma~9]{Matomaki} and \cite{Matomaki2007}). 
\begin{theorem}\label{Theorem-Matomaki}
There exist positive real numbers $d_1 < 1$ and $D$ such that, for every
sufficiently large $x$ and every $\gamma \in [1/2, 1]$, the interval $[x, 2x]$ contains at most
$Dx^{2/3-\gamma}$ disjoint intervals $[n, n + n^\gamma]$ for which
\[
\# ([n,n+n^\gamma] \cap \mathcal{P})  \leq d_1 \frac{n^\gamma}{\log n}.
\]
\end{theorem}

By applying this theorem, Matom\"{a}ki \cite{Matomaki} constructed an element in $\mathcal{W}(2)$ as Theorem~\ref{Theorem:Matomaki}. It is truly astonishing because Mills' construction suggests that we need the existence of a prime number in each interval $[n, n+n^{1/2}]$ to construct an element in $\mathcal{W}(2)$. This problem is strongly connected with Legendre's conjecture which asserts that for every $n$ there exists a prime number between $n^2$ and $(n+1)^2$. It is unsolved and believed to be extremely difficult. However, Theorem~\ref{Theorem-Matomaki} with $\gamma=1/2$ states that there exists a prime number $p \in [n,n+n^{1/2}]$ except for very few intervals $[n,n+n^{1/2}]$. Therefore, Matom\"{a}ki discovered that we can construct a PRC if the exceptional intervals are very few. We apply Theorem~\ref{Theorem-Matomaki} to prove the following lemma.

\begin{lemma}\label{Lemma-ModifiedMatomaki} Let $E$ and $c$ be real numbers with  $2\leq c\leq E$. Let $d_1$ be as in Theorem~\ref{Theorem-Matomaki}. Let $\epsilon\in(0,1)$  and $d_2>0$.  Then, there exists $X_0=X_0(E,\epsilon,d_2)>0$  such that for all real numbers $X\geq X_0$ and $\eta\in [1/2,1-\epsilon]$, if we have
\begin{equation}\label{eq:NumberPrimes1}
\#([X,X+X^\eta] \cap \mathcal{P} ) \geq  \frac{d_2 X^\eta}{\log X}, 
\end{equation}
then there exists $q\in[X,X+X^\eta] \cap \mathcal{P}$ such that 
\begin{equation}\label{eq:NumberPrimes2}
\#( [q^c , q^c + q^{c-1} ] \cap \mathcal{P}) \geq   \frac{d_1 q^{c-1}}{\log q^c}.
\end{equation}
\end{lemma}

\begin{proof}
Let $D$ be as in Theorem~\ref{Theorem-Matomaki}. Let $X_0$ be a sufficiently large parameter depending only on $E$, $\epsilon$, and $d_2$. Take arbitrary real numbers $X\geq X_0$ and $\eta\in [1/2,1-\epsilon]$. We suppose that \eqref{eq:NumberPrimes1} is true. Then, let $R=[X,X+X^\eta] \cap \mathcal{P}$. Each interval $[q^{c},q^{c}+q^{c-1}]$ $(q\in R)$ is completely contained in $[X^{c}, 2 X^{c} ]$. Indeed, we have
\begin{align*}
X^{c} &\leq q^{c} < q^{c}+q^{c-1} = q^c (1+q^{-1}) \leq (X+X^\eta)^c(1+X^{-1})\\
&\leq X^c (1+ X^{\eta-1})^{E}(1+X^{-1}) \leq 2 X^c
\end{align*}
since $\eta\leq 1-\epsilon$, $X\geq X_0$, and $X_0$ is sufficiently large. Setting $\gamma = (c-1)/c$, we have $\gamma \in [1/2,1)$. By $\eta\geq 1/2$ and $c\geq 2$, the following inequalities are true:
\begin{align}\nonumber
\frac{d_2 X^\eta}{\log X}- D X^{c(2/3- \gamma)} &=  \frac{d_2 X^\eta}{\log X}- D X^{1-c/3} = \frac{ X^\eta}{\log X} \left(d_2- DX^{1-\eta-c/3} \log X  \right)\\ \label{eq:NumberPrimes3}
&\geq \frac{{X^\eta}}{\log X} \left(d_2- D X^{-1/6}\log X \right)>0,
\end{align}
since $X\geq X_0$ and $X_0$ is sufficiently large. Then by combining \eqref{eq:NumberPrimes1}, \eqref{eq:NumberPrimes3}, and Theorem~\ref{Theorem-Matomaki} with $x\coloneqq X^c$ and $\gamma\coloneqq \gamma$, there exists a prime number $q\in R$ satisfying \eqref{eq:NumberPrimes2}.
\end{proof}

\begin{proof}[Proof of Lemma~\ref{Lemma-key2}]
Assume that for every $k_0>0$ there exists $k\in \mathcal{I}_b\cap [k_0,\infty)$ such that 
\begin{equation}\label{eq:assumption}
p_{k+1} \notin [ p_{k}^{c_{k+1}}, p_{k}^{c_{k+1}} + p_{k}^{\theta c_{k+1}} ].
\end{equation}
We will find $w\in \mathcal{W}(c_k)$ such that $w<\xi$ by Matom\"{a}ki's construction, which will imply a contradiction to the minimality of $\xi$. 

Let $B$ be as in \eqref{equation:B}, let $\epsilon=\min(1/B,1/4)$, and let $d_2=\min (d_0,d_1)$, where $d_0$ and $d_1$ denote as in Theorems~\ref{Theorem-BakerHarmanPintz} and \ref{Theorem-Matomaki}, respectively. Let $X_0=X_0(B,\epsilon, d_2)$ be as in Lemma~\ref{Lemma-ModifiedMatomaki}. We now take $k_0$ as a sufficiently large positive parameter depending only on $B$, $\epsilon$, and $d_2$. By the assumption \eqref{eq:assumption}, we find an integer $k$ in $\mathcal{I}_b\cap [k_0,\infty)$ satisfying \eqref{eq:assumption}. 
By Lemma~\ref{Lemma-key1}, the inequality $p_{k}^{c_{k+1}}>p_{k+1}$ does not hold, and hence we have
\begin{equation}\label{eq-seed:of:cont}
p_{k}^{c_{k+1}} + p_{k}^{\theta c_{k+1}}<p_{k+1} .
\end{equation}
Then, we set $q_k=p_k$.  Theorem~\ref{Theorem-BakerHarmanPintz} with $x\coloneqq q_k^{c_{k+1}}$ implies that 
\begin{equation}\label{eq:BHP1}
 \# ([q_k^{c_{k+1}}, q_k^{c_{k+1}}+q_k^{\theta c_{k+1}}] \cap \mathcal{P})  \geq \frac{d_0 q_k^{\theta c_{k+1}}}{\log q_k^{c_{k+1}}}.
 \end{equation}
 We may assume that $q_k^{c_{k+1}}\geq X_0=X_0(B,\epsilon, d_2)$ since $k\geq k_0$ and $k_0=k_0(B,\epsilon,d_2)$ is sufficiently large. By \eqref{eq:BHP1} and Lemma~\ref{Lemma-ModifiedMatomaki} with 
\begin{gather*}
E\coloneqq B,\quad c\coloneqq c_{k+2}\in [2,B], \quad \epsilon\coloneqq \epsilon, \quad d_2\coloneqq d_2, \\
X\coloneqq q_k^{c_{k+1}}, \quad \eta\coloneqq \theta\in [1/2,3/4]\subseteq [1/2,1-\epsilon], 
\end{gather*}
there exists $q_{k+1}\in \mathcal{P}$ satisfying
\begin{gather}\label{eq:case:m=k1}
q_k^{c_{k+1}} \leq q_{k+1} \leq q_k^{c_{k+1}}+q_k^{\theta c_{k+1}} (\leq q_k^{c_{k+1}}+q_k^{c_{k+1}-1}), \\ \label{eq:case:m=k2}
\#( [q_{k+1}^{c_{k+2}} , q_{k+1}^{c_{k+2}} + q_{k+1}^{c_{k+2}-1} ] \cap \mathcal{P}) \geq   \frac{d_1 q^{c_{k+2}-1}}{\log q^{c_{k+2}}},
\end{gather}
where the last inequality of \eqref{eq:case:m=k1} follows from $c_{k+1}= b\geq 3$ and $\theta=21/40$.
Let us construct a sequence $(q_{m})_{m=k}^\infty$ of prime numbers satisfying the following properties: for every $m\geq k$, we have
\begin{gather} \label{property-1}
q_m^{c_{m+1}}\leq q_{m+1} \leq q_m^{c_{m+1}}+q_m^{c_{m+1}-1},\\ \label{property-2}
\# ([q_{m+1}^{c_{m+2}},q_{m+1}^{c_{m+2}}+q_{m+1}^{c_{m+2}-1}] \cap \mathcal{P})  \geq  \frac{d_1 q_{m+1}^{c_{m+2}-1}}{\log q_{m+1}^{c_{m+2}}}.
\end{gather}
 In the case $m=k$,  both \eqref{property-1} and \eqref{property-2} are true by \eqref{eq:case:m=k1} and \eqref{eq:case:m=k2}. We assume that there exists a sequence $(q_{m})_{m=k}^M$ of prime numbers satisfying 
\eqref{property-1} and \eqref{property-2} for every $m=k,\ldots , M-1$, where $M\geq k+1$.  Then by applying \eqref{property-2} with $m=M-1$ and Lemma~\ref{Lemma-ModifiedMatomaki} with 
\begin{gather*}
E\coloneqq B, \quad c\coloneqq c_{M+2}\in [2,B],\quad  \epsilon\coloneqq \epsilon, \quad d_2\coloneqq d_2, \\ X\coloneqq q_M^{c_{M+1}}, \quad  \eta\coloneqq (c_{M+1}-1)/c_{M+1}\in [1/2, 1-1/B]\subseteq [1/2,1-\epsilon], 
\end{gather*}
we find that there exists a prime number $q_{M+1}$ satisfying \eqref{property-1} and \eqref{property-2} with $m=M$. By induction, we obtain an infinite sequence $(q_m)_{m=k}^\infty$ of prime numbers satisfying \eqref{property-1} and \eqref{property-2}. 

Furthermore, we set $q_m=p_m$ for every $m=1,2,\ldots, k-1$. By Lemma~\ref{Lemma-key1} and \eqref{property-1}, we obtain 
\[
q_m^{c_{m+1}} \leq q_{m+1} <  (q_{m} +1)^{c_{m+1}}-1
\]
for every $m\in \mathbb{N}$. Therefore, we have 
\begin{equation}\label{eq:prc-inequality}
q_m^{1/C_m} \leq q_{m+1}^{1/C_{m+1}}< (q_{m+1}+1)^{1/C_{m+1}} <(q_{m}+1)^{1/C_{m}}
\end{equation}
for every $m\in \mathbb{N}$, and hence $\lim_{m\to\infty} q_m^{1/C_m}=\lim_{m\to\infty} (q_m+1)^{1/C_m} \eqqcolon w$ exist. In addition, $w\in \mathcal{W}(c_k)$ holds by \eqref{eq:prc-inequality}. On the other hand, combining \eqref{eq-seed:of:cont} and \eqref{eq:case:m=k1}, we have
\[
\lfloor w^{C_{k+1}} \rfloor = q_{k+1}\leq p_k^{c_{k+1}}+p_k^{\theta c_{k+1}} < p_{k+1} = \lfloor \xi ^{C_{k+1}}\rfloor.
\]
This is a contradiction since $\xi$ is the smallest element in $\mathcal{W}(c_k)$.
\end{proof}
\begin{lemma}\label{Lemma-Nearest} There exist $k_1=k_1(\xi)>0$ and $K_0>0$ such that for every integer $k\in \mathcal{I}_b \cap[k_1, \infty)$, we have
\begin{equation}\label{eq:Nearest}
0\leq \xi^{C_k}- p_k\leq K_0 p_k^{-c_{k+1}\theta_b },
\end{equation}
where $\theta_b=1-\theta-1/b>0$. Further, there exists $\gamma>0$ such that for every $k\in \mathcal{I}_b \cap[k_1, \infty)$
\begin{equation}\label{eq:Nearest2}
|\xi^{C_k}- p_k|\leq e^{-\gamma C_{k}}.
\end{equation}

\end{lemma}
\begin{proof}Let $k_0(\xi)$ be as in Lemma~\ref{Lemma-key2}, and let $k_1=k_1(\xi)>k_0(\xi)$ be a sufficiently large parameter. Since $p_{k}= \lfloor \xi^{C_k}\rfloor$ for every $k\in \mathbb{N}$, by Lemma~\ref{Lemma-key2}, each $k \in \mathcal{I}_b\cap [k_1,\infty)$ satisfies
\begin{equation}\label{eq:N1}
p_{k}^{c_{k+1}} \leq p_{k+1} \leq \xi^{C_{k+1}} \leq p_{k+1}+1 \leq p_{k}^{c_{k+1}} + 2p_{k}^{\theta c_{k+1}}.
\end{equation}
Therefore, for every $k \in \mathcal{I}_b\cap [k_1,\infty)$, we have 
\begin{equation}\label{eq:pkA}
p_{k}  \leq \xi^{C_{k}} \leq (p_{k}^{c_{k+1}} + 2p_{k}^{\theta c_{k+1}})^{1/c_{k+1}}= p_k (1+2p_k^{c_{k+1}(\theta-1)} )^{1/c_{k+1}} \leq p_k + 2p_k^{c_{k+1}(\theta -1)+1}.
\end{equation}
We note that $c_{k+1}(\theta -1)+1 = c_{k+1}(\theta -1+\frac{1}{b}) = -c_{k+1}\theta_b<0 $ since $c_{k+1}= b$ and $b\geq 3$. Therefore, we get \eqref{eq:Nearest}. Furthermore, Lemma~\ref{Lemma-key1} leads to 
\begin{equation}\label{eq:decreasing}
p_k^{c_{k+1}} \geq p_{k-1}^{c_{k+1}c_{k}} \geq p_{k-2}^{c_{k+1}c_kc_{k-1}} \geq \cdots \geq p_{1}^{C_{k+1}/c_1}.
\end{equation}
Combining \eqref{eq:Nearest} and \eqref{eq:decreasing}, we obtain \eqref{eq:Nearest2}.
\end{proof}

%\begin{remark}\label{Remark-Matomaki}
%We can observe the main difficulties of Question~\ref{Question-1} in \eqref{eq:pkA}. If $c_1=c_2=\cdots=2$, then roughly speaking, \eqref{eq:pkA} should be
%\[
%p_{k}  \leq \xi ^{2^k}<  p_k + p_k^{2(1/2 -1)+1} =p_k+1.
%\]
%It does not suffice to obtain \eqref{eq:Nearest2}. Further, by Matom\"{a}ki's construction, we do not know which specific prime number $q$ satisfies \eqref{eq:NumberPrimes2}. This is also one of the main difficulties.
%\end{remark}

\section{Proof of Theorem~\ref{Theorem-main1}}\label{Section-Proof}

\begin{lemma}\label{Lemma-trans-or-Pisot} Let $b$ be an integer greater than or equal to $3$. Let $(c_k)_{k=1}^\infty$ be a sequence of integers satisfying \eqref{Condition-c1}, \eqref{Condition-c2}, and \eqref{Condition-c3} in Theorem~\ref{Theorem-main1}. Let $\xi$ be the smallest element in $\mathcal{W}(c_k)$. Then the following properties are true:
\begin{itemize}
\item if $b\geq 5$, then $\xi$ is transcendental;
\item if $b= 4$, then  either $\xi$ is transcendental, or there exists $m\in \mathbb{N}$ such that $\xi^{C_{m}}$ is a Pisot number of degree $2$; 
\item if $b= 3$, then  either $\xi$ is transcendental, or  there exists $m\in \mathbb{N}$ such that $\xi^{C_{m}}$ is a Pisot number of degree $2$ or $3$.
\end{itemize}

\end{lemma}
\begin{proof} Let $k_1$ be as in Lemma~\ref{Lemma-Nearest}. Assume that $\xi$ is a real algebraic number. By the definition of $\mathcal{W}(c_k)$, $\xi>1$ holds.  By Lemma~\ref{Lemma-Nearest},  there exists $\gamma>0$ such that for every  $k\in \mathcal{I}_b \cap[k_1, \infty)$
\[
|\xi^{C_k} -p_k| \leq e^{-\gamma C_{k}}.
\]
Therefore, Theorem~\ref{Theorem-Dubickas1} with $\alpha\coloneqq \xi$, $q\coloneqq 1$, $s_k\coloneqq C_k$, and $\epsilon\coloneqq \gamma$ implies that $ \xi^{C_m}$ is a Pisot number for some $m\in \mathbb{N}$. Let $\beta=\xi^{C_m}$, and let $\ell$ be the degree of $\beta$. We now show the following claim, where $\theta_b$ is defined in Lemma~\ref{Lemma-Nearest}.\\

\noindent\textbf{Claim:} We have $2\leq \ell \leq (b\theta_b)^{-1}+1$.  \\

If $\ell=1$, then $\beta(= \xi^{C_m})$ is a positive integer. Therefore, 
\[
p_{m+1}= \lfloor \xi^{C_{m+1}} \rfloor =(  \xi^{C_{m}})^{c_{m+1}}=p_m^{c_{m+1}}, 
\]
which is a contradiction since $p_{m+1}$ is a prime number and $c_{m+1}\geq 2$. Therefore $\ell \geq 2$. 

We next assume that $\ell > (b\theta_b)^{-1}+1$. Let $\beta_1=\beta, \beta_2,  \ldots ,  \beta_\ell$ be the conjugates over $\mathbb{Q}$ of $\beta$ labeled so that $\beta_1>1>|\beta_2|\geq  \cdots \geq |\beta_\ell|$. For every $k\geq m$, since $C_k/C_m$ is a positive integer, we obtain 
\[
\beta_1^{C_k/C_m} + \sum_{j=2}^\ell  \beta_j^{C_k/C_m} \in \mathbb{Z}.  
\] 
Let $k$ be a sufficiently large integer in $\mathcal{I}_b\cap[m,\infty)$. By Lemma~\ref{Lemma-Nearest}, 
\begin{equation}\label{eq:claim1}
\|\sum_{j=2}^\ell  \beta_j^{C_k/C_m}\| = \|\beta_1^{C_k/C_m}  \|=\|\xi^{C_k}  \|\leq  K_0 p_k^{-c_{k+1}\theta_b}. 
\end{equation}
By \eqref{eq:N1},  we recall that $\xi^{C_{k+1}} \leq  p_k^{c_{k+1}}+2p_k^{\theta c_{k+1}}$, to obtain
\begin{equation}\label{eq:claim2}
\xi^{C_{k+1}} \leq  p_k^{c_{k+1}}(1+2p_k^{(\theta-1) c_{k+1}})\leq 2 p_k^{c_{k+1}}.
\end{equation}
Therefore, combining \eqref{eq:claim1}, \eqref{eq:claim2}, and $c_{k+1}=b$, we have 
\[
\|\sum_{j=2}^\ell  \beta_j^{C_k/C_m}\| \leq  K_1 \xi^{-C_{k+1}\theta_b} = K_1 \beta_1^{-b\theta_b C_{k}/C_m  }  
\]
for some constant $K_1>0$. By Lemma~\ref{lemma-Dubickas2} with $n\coloneqq C_k/C_m$, there exists $\lambda=\lambda(\beta)>0$ such that 
\[
\left\|\sum_{j=2}^\ell  \beta_j^{C_k/C_m}\right\|=\left|\sum_{j=2}^\ell  \beta_j^{C_k/C_m}\right| \geq |\beta_2|^{C_k/C_m} (C_k/C_m)^{-\lambda} 
\]
since $k$ is sufficiently large. By taking $k\to \infty$ on $k\in \mathcal{I}_b$, we conclude that $|\beta_2| \leq \beta_1^{-b\theta_b}$. Since $|\beta_2| \geq |\beta_3|\geq \cdots \geq |\beta_l|$ and $\ell >  (b\theta_b)^{-1}+1$, we have $|\beta_2\cdots \beta_\ell|\leq \beta_1^{-(\ell-1) b\theta_b}< \beta_1^{-1}$. This is a contradiction to $|\beta_1\cdots \beta_\ell| \geq 1$. Thus, we obtain the claim. \\[-5pt]

\noindent\underline{Case} ($b\geq 5$): We now recall that $\theta_b=1-\theta-1/b=19/40-1/b$, and so 
\[
b\theta_b \geq \frac{19\cdot 5}{40}-1 \geq  \frac{19}{8}-1 = \frac{11}{8}\quad \therefore  \quad (b\theta_b)^{-1}+1\leq  \frac{8}{11}+1<2,
\]
a contradiction to the claim. Therefore, $\xi$ is transcendental. 

\noindent\underline{Case} ($b= 4$): By simple calculation, 
\[
b\theta_b = \frac{19\cdot 4}{40}-1 =  \frac{19}{10}-1 = \frac{9}{10}\quad \therefore  \quad (b\theta_b)^{-1}+1= \frac{10}{9}+1<3.
\]

\noindent\underline{Case} ($b=3$): We also observe that
\[
b\theta_b = \frac{19\cdot 3}{40}-1 =  \frac{57}{40}-1 = \frac{17}{40}\quad \therefore  \quad \ (b\theta_b)^{-1}+1= \frac{40}{17}+1<4.
\]
This completes the proof. 
\end{proof}

Let us discuss the case when $\xi^{C_m}$ is a Pisot number of degree $2$ for some $m\in \mathbb{N}$. 

\begin{lemma}\label{Lemma-Trace}
Let  $b$, $(c_k)_{k=1}^\infty$, and $\xi$ be as in Lemma~\ref{Lemma-trans-or-Pisot}. Suppose that there exists $m\in \mathbb{N}$ such that $\xi^{C_m}$ is a Pisot number. Let $\beta= \xi^{C_m}$ and $\beta_1=\beta, \beta_2,\ldots, \beta_\ell$ be the conjugates over $\mathbb{Q}$ of $\beta$. We set 
\[
t_k=\beta_{1}^{C_k/C_m }  +\cdots + \beta_{\ell}^{C_k/C_m }\in \mathbb{Z}
\]
for every $k\geq m$. Then we have $t_k=\lfloor \xi^{C_k} \rfloor =p_k$ for every sufficiently large $k\in \mathbb{N}$. 
\end{lemma}
\begin{proof}
Let $\delta_k = \sum_{j=2}^\ell \beta_j^{C_k/C_m }$ for every $k\geq m$. Let $k$ be a sufficiently large integer in $\mathcal{I}_b\cap [m,\infty)$. Since $t_k$ and $\beta_{1}^{C_k/C_m } $ are real numbers, $\delta_k$ is also a real number. Further, we recall that $|\beta_j|<1$ for every $2\leq j\leq \ell$, which implies that $\delta_k\in (-1/2,1/2)$ since $k$ is sufficiently large. If $0<\delta_k<1/2 $, then by \eqref{eq:Nearest}, $0\leq \xi^{C_{k}} - p_k \leq K_0 p_k^{-c_{k+1}\theta_b}$ holds for some constant $K_0>0$. Therefore, 
\[
\delta_k\leq t_k -p_k \leq \delta_k + K_0 p_k^{-c_{k+1}\theta_b}.
\]
Since we have $\delta_k>0$, $t_k -p_k\in \mathbb{Z}$, and $-c_{k+1}\theta_b<0$, this is a contradiction by choosing $k$ large enough.  Thus, $-1/2<\delta_k\leq 0 $ holds and $\xi^{C_k}=\beta_{1}^{C_k/C_m }=t_k -\delta_k$, and hence $t_k =\lfloor \xi^{C_k} \rfloor=p_k$. \end{proof}

\begin{lemma}\label{Lemma-not-Pisot}
Let $b$, $(c_k)_{k=1}^\infty$,  and $\xi$ be as in Lemma~\ref{Lemma-trans-or-Pisot}.  If $b\in \{3,4\}$, then there does not exist $m\in \mathbb{N}$ such that $\xi^{C_{m}}$ is a Pisot number of degree $2$. 
\end{lemma}

\begin{proof}
Assume that there exists $m\in \mathbb{N}$ such that $\xi^{C_{m}}$ is a Pisot number of degree $2$. Let $\beta=\xi^{C_m}$, and let $\beta_1=\beta, \beta_2$ be the conjugates of $\beta$. Let $t_k$ be as in Lemma~\ref{Lemma-Trace}. \\[-5pt]

\noindent\underline{Case} ($b=3$): For every $k\in \mathcal{I}_3$,  we set $x_1=\beta_{1}^{C_k/C_m}$ and $x_2=\beta_2^{C_k/C_m}$. Then we obtain $c_{k+1}=3$ and
\[
t_k^3 =(x_1+x_2)^3= x_1^3 +3 x_1^2 x_2 +3 x_1 x_2^2 +x_2^3 = t_{k+1} + 3x_1x_2 t_k. 
\]
 By applying Lemma~\ref{Lemma-Trace}, for sufficiently large $k\in \mathcal{I}_3$, we have $p_k^3 = p_{k+1} + 3x_1x_2 p_k$. By $x_1x_2 \in \mathbb{Z}$, we see that $p_{k+1}$ is divisible by $p_k$, a contradiction since $p_k$ and $p_{k+1}$ are distinct prime numbers. \\[-5pt]
 
\noindent\underline{Case} ($b=4$): Let $k$ be a sufficiently large integer in $\mathcal{I}_4\cap[m,\infty)$.  Since $\mathcal{I}_4$ is infinite, we may assume that $C_k/C_m(\eqqcolon d_k)$ is a positive even integer. Setting $x_1=\beta_{1}^{d_k}$ and $x_2=\beta_2^{d_k}$, we have 
\[
\beta_1^{d_k}+ \beta_2^{d_k}=  p_k = \lfloor \beta_1^{d_k}  \rfloor \leq \beta_1^{d_k}. 
\]
Then $\beta_2^{d_k} \leq 0$, a contradiction since $\beta_2\in \mathbb{R}$ and $d_k$ is a positive even number. 
\end{proof}

Combining Lemmas~\ref{Lemma-trans-or-Pisot} and \ref{Lemma-not-Pisot} completes the proof of Theorem~\ref{Theorem-main1}.

\begin{remark}\label{Remark-b=3}
To prove the transcendency of Mills' constant, it remains to show that there does not exist $m\in \mathbb{N}$ such that $\xi^{C_{m}}$ is a Pisot number of degree $3$ when $b=3$. In this remark, we describe the difficulty of this case. 

Suppose $b=3$, and we assume that there exists $m\in \mathbb{N}$ such that $\xi^{C_{m}}$ is a Pisot number of degree $3$.  Let  $\beta=\xi^{C_{m}}$ and $\beta_1=\beta, \beta_2,\beta_3$ be the conjugates of $\beta$. Similarly with the proof of Lemma~\ref{Lemma-not-Pisot}, we obtain 
\begin{equation}\label{eq:Pisotdeg3}
p_{k+1} = p_{k}^3-3 b_k p_k +3 e_k
\end{equation}
for every sufficiently large $k\in \mathcal{I}_3$, where 
\begin{gather*}
b_k\coloneqq (\beta_1\beta_2)^{C_k/C_m}+(\beta_2\beta_3)^{C_k/C_m} +(\beta_3\beta_1)^{C_k/C_m}, \quad e_k\coloneqq (\beta_1\beta_2 \beta_3)^{C_k/C_m}.
\end{gather*} 
Therefore,  we have to investigate properties of $b_k$ and $e_k$. The author does not have any good ideas on how to treat $b_k$ and $e_k$ simultaneously. Indeed, the method of the proof of  Lemma~\ref{Lemma-not-Pisot} does not work. 
\end{remark}

\section{Further discussions}\label{Section-Further}

We give several assumptions leading to the transcendency of $\xi_3$. 
\begin{proposition}\label{Proposition-mod3} Assume that $\lfloor \xi_3 ^{3^k}\rfloor \not\equiv \lfloor \xi_3 ^{3^{k+1}}\rfloor \mod 3$ for infinitely many positive integers $k$. Then $\xi_3$ is transcendental. 
\end{proposition}
\begin{proof}
Suppose that $\xi_3$ is algebraic. Then by Theorem~\ref{Theorem-Mills}, there exists $m\in \mathbb{N}$ such that $\xi_3^{3^m}$ is a Pisot number of degree $3$. Let $p_k=\lfloor \xi_3^{3^k} \rfloor$ for every $k\in \mathbb{N}$. By Fermat's little theorem, $p_k^3\equiv p_k \mod 3$. Thus, by Remark~\ref{Remark-b=3}, we have $p_{k+1} \equiv p_{k} \mod 3$ for every sufficiently large $k\geq m$. 
\end{proof}
A \textit{Mills prime} is a prime number of the form $\lfloor \xi_3^{3^k}\rfloor$. By Proposition~\ref{Proposition-mod3}, if there are infinitely many Mills primes of the forms $3n+1$ and $3n+2$ respectively, then $\xi_3$ is transcendental. However, it seems to be impossible to classify Mills primes modulo $3$.    

\begin{proposition}\label{Proposition-PrimeGap} Assume that there exist $\theta'\in (0,1/2)$ and $x_0>0$ such that for every $x\geq x_0$ we can find a prime number $p\in [x, x+x^{\theta'}]$. Then $\xi_3$ is transcendental. 
\end{proposition}
\begin{proof}
Suppose that $\xi_3$ is algebraic. Then by Theorem~\ref{Theorem-Mills}, there exists $m\in \mathbb{N}$ such that $\xi_3^{3^m}$ is a Pisot number of degree $3$. We set $\theta_b'=\theta_3'=1-\theta'-1/3$. In a similar manner with Lemma~\ref{Lemma-Nearest} and the claim in the proof of Lemma~\ref{Lemma-trans-or-Pisot}, we have
\[
3=\deg (\xi_3^{3^m})\leq  (3\theta_3')^{-1}+1 = \frac{1}{2-3\theta'} + 1< \frac{1}{2-3/2}+1 = 3,
\] 
a contradiction.
\end{proof}

Let $p_n$ be the $n$th prime number for every $n\in \mathbb{N}$. Cram\'er \cite{Cramer} conjectured that $p_{n+1}-p_n =O((\log p_n)^2)$. If this conjecture is true, we can take $\theta'$ in Proposition~\ref{Proposition-PrimeGap} as an arbitrarily small positive real number. Therefore, assuming Cram\'er's conjecture, $\xi_3$ is transcendental. 

\begin{question}
Assuming the Riemann hypothesis, is Mills' constant transcendental?
\end{question}
Carneiro, Milinovich, and, Soundararajan \cite[Theorem~1.5]{CMS} showed that  for all real numbers $x\geq 4$, there exists a prime number $p\in \left(x,x+\frac{22}{25}\sqrt{x}\log x\right)$ if the Riemann hypothesis is true. However, this does not reach the assumption in Proposition~\ref{Proposition-PrimeGap}. We need more arithmetic (or algebraic) properties of Mills primes to answer the question affirmatively.

\section*{Acknowledgement}
The author would like to thank Professor Shigeki Akiyama for daily attractive discussions on mathematics. The author is grateful to Professor Hajime Kaneko for giving useful ideas to simplify the proof on the cases $b=3$ and $4$. The author is also grateful to Professor Wataru Takeda and the referees for finding errors and mistakes. The author was supported by JSPS KAKENHI Grant Numbers JP22J00025, JP22KJ0375, and JP25K17223. 

\bibliographystyle{abbrv}
\bibliography{references}

\end{document}